\documentclass{amsart}
\newtheorem{theorem}{Theorem}[section]
\newtheorem{lemma}[theorem]{Lemma}
\newtheorem{corollary}[theorem]{Corollary}
\newtheorem{proposition}[theorem]{Proposition}
\newtheorem{algorithm}[theorem]{Algorithm}
\theoremstyle{definition}
\newtheorem{definition}[theorem]{Definition}
\newtheorem{example}[theorem]{Example}

\theoremstyle{remark}
\newtheorem{remark}[theorem]{Remark}
\numberwithin{equation}{section}
\begin{document}
\title{Quasi-Invariant Measures, Escape Rates\\ And The Effect Of The Hole}
%    Information for first author
\author{Wael Bahsoun}
    %Address of record for the research reported here
    \address{Department of Mathematical Sciences, Loughborough University, 
Loughborough, Leicestershire, LE11 3TU, UK}
\email{W.Bahsoun@lboro.ac.uk}
\author{Christopher Bose}
    %Address of record for the research reported here
\address{Department of Mathematics and Statistics, University of Victoria,  
   PO BOX 3045 STN CSC, Victoria, B.C., V8W 3R4, Canada}
\email{cbose@uvic.ca}
    %\thanks will become a 1st page footnote.
\thanks{W.B. thanks the Department of Mathematics and Statistics at the University of Victoria, Canada, for hosting him during May-June 2009.  His visit was supported by by Loughborough University Small Faculty Grant  number H10621.}
\date{\today}
%\dedicatory{This paper is dedicated to our authors.}
\begin{abstract}
Let $T$ be a piecewise expanding interval map and $T_H$ be an abstract perturbation of $T$ into an interval map with a hole. Given a number $\ell$, $0<\ell<1$, we compute an upper-bound on the size of a hole needed for the existence of an absolutely continuous conditionally invariant measure (accim) with escape rate not greater than $-\ln(1-\ell)$. The two main ingredients of our approach are Ulam's method and an abstract perturbation result of Keller and Liverani.
\end{abstract}
\maketitle
\pagestyle{myheadings} 
\markboth{Quasi-Invariant Measures, Escape Rates And The Effect Of The Hole}{W. Bahsoun And C. Bose}
\section{Introduction}
Open dynamical systems have  recently been a very active topic of research in ergodic theory and dynamical systems. Such dynamical systems are used in studying nonequilibrium statistical mechanics \cite{Ru} and metastable chaos \cite{YY}. 

\bigskip

A dynamical system is called open if there is a subset in the phase space such that whenever an orbit lands in it, the dynamics of this obit is terminated; i.e, the orbit dies or disappears. The subset through which orbits escape is called a hole, denoted $H$. The escape rate through $H$ can be measured if the system admits an absolutely  continuous conditionally invariant measure (accim).  The first result  in this direction is due to Pianigiani and Yorke \cite{PY}. The survey article \cite{DY} contains a considerable list of references on the existence of accim and its relation to other measures. One of the most intuitive existence results is found in Section 7 of \cite{LM}. It is mainly concerned with systems having small holes and its idea is based on the perturbation result of \cite{KL}. It roughly says that if a mixing interval map is perturbed by introducing a `sufficiently small' hole, then the resulting open dynamical system admits an accim. Our main goal in this paper is to show how the condition `sufficiently small'  can be computationally verified in some of these results, in particular, results from  Section 7 of \cite{LM}.

\bigskip

More precisely, for a given Lasota-Yorke map $T$, we use Ulam's method on the closed dynamical system $T$ to give a computable size of the hole $H$ for which the open dynamical system $T_H$ 
must admit an accim.

\bigskip
Historically, Ulam approximations have been used to provide rigorous estimates of invariant densities of closed systems (see \cite{Li} and references cited there) or to approximate other dynamical invariants (see \cite{Fr}). The use of Ulam's method in the study of open systems is natural. In \cite{WB}, Ulam approximations were used to rigorously estimate the escape rate for certain open systems. In that computation the Ulam matrix was derived from 
the Perron-Frobenius operator associated to the open system (i.e., a sub-stochastic 
matrix), whereas here, we approximate the closed system.
The method in \cite{WB} also demanded existence of an accim as a basic assumption.
\bigskip

We remark that, as a consequence of the spectral methods discussed here, upper bounds on the 
escape rate can be obtained from analysis of the closed system. However this does not generally
replace the computation in \cite{WB} where the Ulam approximation of the
open system yields both upper and lower bounds on the escape rate (but under 
the additional assumption of existence of an accim).
Hence, there
is potential to apply a two step method -- the current algorithm would be
used to guarantee an accim and 
to provide rough (upper) bounds on the escape rate, followed by 
the method of \cite{WB}, once the size of the hole 
is fixed, to more accurately estimate the latter.
\bigskip

Our paper is organized as follows.
In Section 2 we present a version of  Keller and Liverani's abstract perturbation theorem. The constants which are involved in this theorem are essential in all our computations and thus, we need to state all the details of this theorem explicitly.  Section 3 contains a precise setting of the problem. Section 4  contains technical lemmas, mostly well-known and stated without proof.
 Section 5 presents the algorithm (Algorithm \ref{alg}) whose 
outputs ($\delta_{com}, \varepsilon_{com}$) are the parameters used to solve our problem.  
Theorem \ref{Th2}, takes these parameters and computes a maximum hole size 
leading to an accim.  In Section 6 we provide, in detail, rigorous computations of the size of a hole for two examples as well as a discussion of computational overhead and some techniques for 
reducing computation time.  In Section 7 we discuss how our methods can be implemented in a smooth setting where there are interesting results concerning the effect of the position of a hole on the escape rate. This is in connection with the recent results of \cite{BY} and \cite{KL2}.  The examples of Section \ref{comp} and the discussion of Section \ref{smooth} highlight the new results that Algorithm \ref{alg} brings to open dynamical systems.  
\section{The First Keller-Liverani Perturbation Result}
Let $(I,\mathfrak B,\lambda)$ be the measure space where $I=[0,1]$, $\mathfrak B$ is the Borel 
$\sigma$-algebra and $\lambda$ is Lebesgue measure. 
Let $L^{1}=L^{1}(I,\mathfrak B,\lambda)$. For $f\in L^1$, we define 
$$Vf=\inf_{\overline f}\{\text{var}\overline f\, : f=\overline f \text{ a.e.}\},$$
where
$$\text{var}\overline f=\sup\{\sum_{i=0}^{l-1}|\overline{f}(x_{i+1})-\overline{f}(x_i)|\, :0=x_0<x_1<\dots<x_l=1\}.$$
We denote by $BV$ the space of functions of bounded variation on $I$ equipped with 
the norm $\| \cdot\| _{BV}=V(\cdot)+\| \cdot\| _{1}$ \cite{DS}.  Let $P_i:BV(I)\to BV(I)$  be two bounded linear operators, $i=1,2$. 
We assume that:
\noindent For $f\in L^1$ 
\begin{equation}\label{E0} 
\| P_if\| _1\le \| f\| _1,
\end{equation}
and $\exists\,\alpha\in (0,1)$, $A>0$ and $B\ge 0$ such that
\begin{equation}\label{E1} 
\| P_{i}^nf\| _{BV}\le A\alpha^n\| f\| _{BV}+B\| f\| _1 \hskip .3cm \forall n\in\mathbb N\,\ \forall f\in BV(I),\, i=1,2.
\end{equation}
Further, we introduce the mixed operator norm:
$$|||P_i|||=\underset{\| f\| _{BV}\le 1}{\sup}\| P_if\| _1.$$
For any bounded linear operator $P: BV\to BV$ with spectrum $\sigma(P)$, consider the set
$$V_{\delta,r}(P)=\{z\in\mathbb C: |z|\le r\text{ or dist}(z,\sigma(P))\le\delta\}.$$
Since the complement of $V_{\delta,r}(P)$ belongs to the resolvent of $P$, it follows that 
(\cite{DS} Lemma 11, VII.6.10)
$$H_{\delta,r}(P)=\sup_{z}\{\| (z-P)^{-1}\| _{BV}: z\in\mathbb C\setminus V_{\delta,r}\}<\infty.$$
\begin{remark}
$\alpha$ in (\ref{E1}) is an upper bound on the essential spectral radius of $P_i$ \cite{Ba}. 
\end{remark}
\begin{theorem}\cite{KL,Li}\label{Th1}
Consider two operators $P_i: BV(I)\to BV(I)$ which satisfy (\ref{E0}) and (\ref{E1}). For $r\in(\alpha, 1)$, let 
$$n_1=\lceil \frac{\ln 2A}{\ln r/\alpha}\rceil$$
$$C=r^{-n_1};\hskip 1cm D=A(A+B+2)$$
$$n_2=\lceil\frac{\ln 8BDCH_{\delta,r}(P_1)}{\ln r/\alpha}\rceil .$$
If
$$|||P_1-P_2|||\le\frac{r^{n_1+n_2}}{8B(H_{\delta,r}(P_1)B+(1-r)^{-1})}\overset{\textrm{def}}{=}
\varepsilon_1(P_1,r,\delta)$$
then for each $z\in\mathbb C\setminus V_{\delta,r}(P_1)$, we have
$$\| (z-P_2)^{-1}f\| _{BV}\le \frac{4(A+B)}{1-r}r^{-n_1}\| f\| _{BV}+\frac{1}{2\varepsilon_1}\| f\| _1.$$
Set
$$\gamma=\frac{\ln(r/\alpha)}{\ln(1/\alpha)},$$ 
$$a=\frac{8[2A(A+B)+(1-r)^{-1}](A+B)^{2}r^{-n_1}+1}{1-r}$$
and
$$b=2[(4(A+B)^2(D+B)+B)(1-r)^{-1}r^{-n_1}+B].$$
If
\begin{equation}\label{E2}
\begin{split} 
|||P_1-P_2|||&\le\min\{\varepsilon_1(P_1,r,\delta),\left[\frac{r^{n_1}}{4B\left(H_{\delta,r}(P_1)
(D+B)+2A(A+B)+(1-r)^{-1}\right) }\right]^{\gamma}\}\\
&\overset{\textrm{def}}{=}\varepsilon_0(P_1,r,\delta)
\end{split}
\end{equation}
then for each $z\in\mathbb C\setminus V_{\delta,r}(P_1)$, we have
\begin{equation}\label{E3}
|||(z-P_2)^{-1}-(z-P_{1})^{-1}|||\le |||P_1-P_2|||^{\gamma}(a\| (z-P_1)^{-1}\| _{BV}+b\| (z-P_1)^{-1}\| _{BV}^2). 
\end{equation}
\end{theorem} 
\begin{corollary}\cite{KL,Li}\label{Co1}
If $|||P_1-P_2|||\le\varepsilon_1(P_1,r,\delta)$ then $\sigma(P_2)\subset V_{\delta,r}(P_1)$. In addition, if 
$|||P_1-P_2|||\le\varepsilon_0(P_1,r,\delta)$, then in each connected component of $V_{\delta,r}(P_1)$ that does not contain $0$ both 
$\sigma(P_1)$ and $\sigma(P_2)$ have the same multiplicity; i.e., the associated spectral projections have the 
same rank.
\end{corollary}
%%%%%%%%%%%BEGIN SECTION 3 %%%%%%%%%%%%%%%%
\section{Expanding Interval Maps, Perturbations and Holes} 
Let $T$ be  a non-singular interval map and denote by $P$ the Perron-Frobenius operator associated 
with $T$ \cite{Ba}. Typically, and this will be the case for our examples, $T$ will be a Lasota-Yorke map\footnote{
A map $T$ acting from an interval $[a,b]$ to itself is a \emph{Lasota-Yorke map} if
it is piecewise $C^2$ with respect to a finite partition $a=x_0 < x_1 < \dots < x_n =b$, has 
well-defined left and right limits of derivatives up to second order at each $x_i$ and is expanding:
 $\beta:=\inf_{x \neq x_i} |T^\prime(x) | > 1$.}.  
We assume:

\bigskip

\noindent (A1) $\exists\,\ \alpha_0\in (0,1)$, and $B_0\ge 0$ such that $\forall f\in BV(I)$
$$VP f\le\alpha_0 Vf+B_0\| f\| _1.$$ 
\begin{remark}\label{R}
Condition (A1) implies that $\rho=1$ is an eigenvalue of $P$. In particular,  $T$ admits an absolutely continuous invariant measure \cite{Ba, BG}. Moreover, for any $r\in(\alpha_0,1)$ there exists a 
$\overline{\delta}_0>0$, $\overline{\delta}_0$ depends on $r$, such that for any $\delta_0\in(0,\overline{\delta}_0]$ and any 
eigenvalue $\rho_i$ of $P$, with $|\rho_i|>r$, we have:
\begin{enumerate}
\item $B(\rho_i,\delta_0)\cap B(0,r)=\emptyset$;
\item $B(\rho_i,\delta_0)\cap B(\rho_j,\delta_0)=\emptyset$, $i\not= j$.
\end{enumerate}
\end{remark}
\noindent The inequality of assumption (A1) is known as a
Lasota-Yorke inequality\footnote{In fact, (A1) is slightly stronger than the original Lasota-Yorke inequality. In particular, when  $T$ is a general piecewise expanding $C^2$ map the Lasota-Yorke inequality is given by $VPf\le2\beta^{-1} Vf+B_0\| f\| _1$.  See \cite{BG} for details and for generalizations of the original result of \cite{LY}. In certain situations, in particular, when $T$ is piecewise expanding and piecewise onto or when $\inf_x|T'(x)|>2$, the original Lasota-Yorke inequality reduces to (A1). In principle, when dealing with Ulam's scheme for Lasota-Yorke maps, as we do in this paper,  (A1) cannot be relaxed. See \cite{Mu} for details.}.  For a given Lasota-Yorke map with $\beta >2$,  the constant $\alpha_0 = 2/\beta $ in 
inequality (A1) and  $B_0$ may be found in terms of bounds on the second derivative
of the map $T$ and the minimum of $x_{i+1} - x_i$. 
\subsection{Ulam's approximation of $P$} 
Let $\eta$ be a finite partition of $I$ into intervals. Let $\text{mesh}(\eta)$ be the mesh size of $\eta$; i.e, the maximum length of an interval in $\eta$, 
and let $\mathfrak B_{\eta}$ be the finite $\sigma$-algebra associated with $\eta$. 
For $f\in L^1$, let $$\Pi_{\eta}f=\mathbb E(f|\mathfrak B_{\eta}),$$
where $\mathbb E(\cdot|\mathfrak B_{\eta})$ denotes the conditional expectation with respect to 
$\mathfrak B_{\eta}$. Specifically, if $x\in I_{\eta}\in\eta$
$$(\Pi_{\eta}f)(x)=\frac{1}{\lambda(I_{\eta})}\int_{I_{\eta}}fd\lambda.$$
Define
$$P_{\eta}=\Pi_{\eta}\circ P\circ\Pi_{\eta}.$$
$P_{\eta}$ is called Ulam's approximation of $P$. Using the basis $\{\frac{1}{\lambda(I_\eta)}\chi_{I_\eta}\}$ in 
$L^1$, $P_{\eta}$ can be represented by a (row) stochastic matrix acting on vectors from 
$\mathbb R^{d(\eta)}$ by right multiplication: $x \rightarrow xP_\eta$. The entries of Ulam's matrix are given by:
$$P_{I_{\eta}J_{\eta}}=\frac{\lambda(I_{\eta}\cap T^{-1} J_{\eta})}{\lambda(I_{\eta})}.$$
Since $P_{\eta}$ can be represented as stochastic matrix, it has a dominant eigenvalue $\rho_{\eta}=1$ \cite{LT}. Any associated left eigenvectors represent invariant functions $f_{\eta}\in BV(I)$ for the 
operator $P_{\eta}$.
\subsection{Interval maps with holes}
Let $H\subset I$ be an open interval. Denote by $T_{H}\overset{\text{def}}{:=}T_{|X_0}$, where $X_0=I\setminus H$. We call $T_H$ an interval map with a hole; $H$ being the hole. Its Perron-Frobenius operator, which we denote by $P_H$, is defined as follows: for $f\in L^1$ and $n\ge 1$
$$P^n_Hf=P^n(f\chi_{X_{n-1}}),$$
where $X_{n-1}=\cap_{i=0}^{n-1} T^{-i} X_0$, the set of points whose orbits do not meet
the hole $H$ in the first $n-1$ steps.
\begin{definition}\label{Def1}
A probability measure $\mu$ on $[0,1]$, $d\mu=f^*_Hd\lambda$, is said to be an absolutely continuous conditionally 
invariant measure (accim) if there exists $0<e_H<1$ such that $P_Hf^*_H=e_H f^*_H$.
In this case $-\ln e_H$ is the escape rate associated to $\mu$.
\end{definition}

%++++++++++++++++++BEGIN SECTION 4 ++++++++++++++
\section{Some Technical Lemmas}
We state (for the most part, without proof) some
well-known technical results to be used in 
our computations later in the paper.

\subsection{Lasota-Yorke inequalities and estimates on the difference of operators in the mixed norm}  

Lasota-Yorke inequalities for $P$ and $P_\eta$ are standard results
from the literature (see \cite{TYLi} for an original source).  The inequality for $P_H$ is not
so well-known, but it is straightforward and we  
derive it below for completeness. 
 \begin{lemma}\label{le2}
The operators $P$ and $P_{\eta}$  satisfy a common Lasota-Yorke inequality as follows:
 $\forall n\in\mathbb N\,\ \forall f\in BV(I)$
$$\|P^nf\| _{BV}\le \alpha_0^n\| f\| _{BV}+\hat B\| f\| _1$$
$$\| P_{\eta}^nf\| _{BV}\le \alpha_0^n\| f\| _{BV}+\hat B\| f\| _1$$
with $\hat B=1+\frac{B_0}{1-\alpha_0}$, independent of $\eta$.  For $T_H$, the map with a hole, and 
under the stronger assumption $\alpha_0<1/3$\footnote
{The assumption that $\alpha_0<1/3$ can be relaxed simply to $\alpha_0<1$. This relaxation can still produce a common Lasota-Yorke inequality for $P$, $P_{\eta}$ and $P_H$ with constants $A$ and $B$ independent of $H$. A common Lasota-Yorke inequality for $P$ and $P_H$ with $\alpha_0<1$ can be found in section 7 of \cite{LM}.  Hence the results of this paper are still valid for $1/3<\alpha_0<1$ using the appropriate Lasota-Yorke inequaility. For the purpose of the examples which we want to present in Section \ref{comp}, it is more sensible to use the constants produced by Lemma \ref{le2}.} 
we have:  $\forall n\in\mathbb N\,\ \forall f\in BV(I)$
$$\| P_H^nf\| _{BV}\le \alpha^n\| f\| _{BV}+B\| f\| _1,$$
where $\alpha=3\alpha_0<1$ and $B=\frac{2\alpha_0+B_0}{1-\alpha}$ with constants
independent of $H$.
\end{lemma}
\begin{proof} For the operator $P_H$, 
by assumption (A1), for $f\in BV(I)$ we have, 
\begin{equation}\label{LY:hole}
\begin{split}
VP_Hf&=VP(f\chi_{X_0})\le\alpha_0V(f\chi_{X_0})+B_0\| f\| _1\\
%&\le\alpha_0(Vf+\sup_{x\in I} f(x)V\chi_{X_0})+B_0\| f\| _1\\
&\le\alpha_0(Vf+2\sup_{x\in I}f(x))+B_0\| f\| _1\\
&\le\alpha_0(Vf+2Vf+2\| f\| _1)+B_0\| f\| _1\\
&=\alpha Vf+(2\alpha_0+B_0)\| f\| _1.
\end{split}
\end{equation}
Therefore, 
$$VP_H^{n}f\le\alpha^{n}Vf+(\sum_{k=1}^{n}\alpha^{k-1})(2\alpha_0+B_0)\| f\| _1,$$ 
and consequently, for all $n\ge 1$
$$\| P_H^nf\| _{BV}\le\alpha^n\| f\| _{BV}+B\| f\| _1.$$
 \end{proof}

\begin{remark}\label{Re:LY}
Since $\alpha_0 < \alpha$ and $\hat B < B$ in the above lemma, we can obtain 
a common Lasota-Yorke inequality, independent of $H$ and $\eta$ using coefficients
$\alpha$ and $B$: 
$$\| P_*^n f \|_{BV} \leq \alpha^n \| f \|_{BV} + B \| f\|_1$$
where $P_*$ represents any of the three operators under discussion. 
For ease of exposition, we will use this common inequality 
in what follows. (However, see Example \ref{Ex2} in Section \ref{comp} for a discussion indicating 
how using the strongest possible inequality can significantly reduce 
computational overhead.)
\end{remark}
\begin{lemma}\label{le2'}
Let $\Gamma=\max\{\alpha_0+1,B_0\}$ and $\varepsilon=\text{mesh}(\eta)$. 
\begin{enumerate}
\item $|||P_{\eta}-P|||\le \Gamma\varepsilon$.
\item If $\lambda(H)\le \Gamma\varepsilon$ then $|||P_{\eta}-P_H|||\le2\Gamma\varepsilon$.
\end{enumerate}
\end{lemma}
\begin{proof}
The first statement is standard. For the proof of the second statement, let $f\in BV(I)$ and observe that
\begin{equation*}
\begin{split}
\| (P_{\eta}-P_H)f\| _1&\le\| (P_{\eta}-P)f\| _1+\| (P-P_H)f\| _1\\
&\le \varepsilon\Gamma \| f\| _{BV}+ \lambda(H)\|f \| _{BV}\le 2\varepsilon\Gamma \| f\| _{BV}.
\end{split}
\end{equation*}
\end{proof}
\subsection{Computer-assisted estimates on the spectrum of P} 
All the constants arising in Theorem \ref{Th1} are (in principle) computable for the 
finite-dimensional operator $P_{\eta}$\footnote{In fact, more precisely, 
we make our computations on the matrix representation of  $P_{\eta}$
acting on the basis $\{ \frac{1}{\lambda(I_\eta)}\chi_{I_\eta}\}$  as in Section 3.}.   Thus, as
proposed in  \cite{Li}, we are going to apply Theorem \ref{Th1} with $P_{\eta}$ as $P_1$ and $P$ as the perturbation $P_2$. This entails some \emph{a priori} estimates.  
\begin{lemma}\label{Le3}
Given $P$, $\delta>0$ and $r\in(\alpha,1)$, there exists $\varepsilon_2>0$ such that for each $\eta$ with 
$0<\text{mesh}(\eta)\le\varepsilon_2$, we have
\begin{equation}\label{E4}
\text{mesh}(\eta)\le(2\Gamma)^{-1}\varepsilon_0(P_{\eta},r,\delta),
\end{equation}
and
\begin{equation}\label{E4'}
|||P_{\eta}-P|||\le\frac{1}{2}\varepsilon_0(P_{\eta},r,\delta).
\end{equation}
\end{lemma}
\begin{proof}
See Lemma 4.2 of \cite{Li}.
\end{proof}
The computation of a lower bound on $\varepsilon_0(P_{\eta},\delta,r)$ follows the 
argument in 
Lemma 3.10 of \cite{WB} the only difference arising from the fact that
here, $P$ is associated to 
$T$ and not $T_H$ as in \cite{WB}.  The key idea is to estimate the 
BV-norm of the resolvent of $P_\eta$ (difficult to compute) by 
the $\| \cdot \|_1-$norm.  
Hence, following Lemma 3.10 of \cite{WB} we define
$$H^*_{\delta,r}(P_{\eta})=\sup\{(\frac{B_0}{r-\alpha_0}+1)\| (z-P_{\eta})^{-1} v\| _1+\frac{1}{r-\alpha_0}+\frac{2}{r}:
 \| v\| _{1}=1, z\in\mathbb C\setminus V_{\delta, r}(P_{\eta})\},$$
 \begin{equation*}
\begin{split}
\varepsilon_0^*(P_{\eta},\delta,r)&\overset{\text{def}}{=}
\min\{\frac{r^{n_1+\lceil\frac{\ln 8BDCH^*_{\delta,r}(P_\eta)}{\ln r/\alpha}\rceil}}{8B(H^*_{\delta,r}
(P_{\eta})B+(1-r)^{-1})},\\
&\hskip 3 cm\left[\frac{r^{n_1}}{4B\left(H^*_{\delta,r}(P_{\eta})(D+B)+2(1+B)+(1-r)^{-1}\right)}\right]^{\gamma}\}.
\end{split}
\end{equation*}
\begin{lemma}\label{Pro1}\text{ }
\begin{enumerate}
\item $\varepsilon_0^*(P_{\eta},\delta,r)$ is uniformly bounded below;
\item $\varepsilon_0^*(P_{\eta},\delta,r)\le \varepsilon_0(P_{\eta},\delta,r)$; 
\item $\text{mesh}(\eta)\le(2\Gamma)^{-1}\varepsilon_0^*(P_{\eta},\delta,r)$ implies $\text{mesh}(\eta)$ satisfies 
(\ref{E4}).
\end{enumerate}
\end{lemma}
 \begin{proof}
Follow the proof of Lemma 3.12 \cite{WB} verbatum. 
\end{proof}
\section{Main result}
Now we have our tools ready to use the computer and rigorously solve the following problem: 
Given a map $T$ satisfying (A1) and given a  number $\ell>0$ compute a number $\varepsilon>0$ such that if
$\lambda(H) < \Gamma \varepsilon$ then the map $T_H$ has an accim with escape rate 
$-\ln e_H < -\ln( 1 - \ell)$.

The critical step is to obtain control on the separation of the point spectrum of $P$ outside 
the essential spectral radius $\alpha$. Naturally, from a computational viewpoint we can 
only really do this for $P_\eta$ after which we use Theorem \ref{Th1} to transfer the
picture to the spectrum of $P$. 

More precisely, the following algorithm will, given the number $\ell$ with  $0< \ell< 1-\alpha$, compute a number $\delta = \delta_{com}$ with $0<\delta<\ell$ and $\varepsilon =\varepsilon_{com} > 0$ such that with $r=1-\ell$, and any $\eta$ with $\text{mesh}(\eta)< \varepsilon$
\begin{enumerate}
\item $\text{mesh}(\eta)\le(2\Gamma)^{-1}\varepsilon_0(P_{\eta},r,\delta)$;
\item $B(1,\delta)\cap B(\rho_i,\delta)=\emptyset$, whenever $\rho_i$ is an eigenvalue of $P_{\eta}$, $\rho_i\notin B(1,\delta)$ and $|\rho_i|>r$.
\end{enumerate}
Thus we obtain the required spectral separation (near the eigenvalue 1) for $P_\eta$ as well as 
the conditions necessary to
apply Theorem \ref{Th1}. 
\begin{algorithm}\label{alg} $T$ and $\ell$ given as above, then
\begin{enumerate}
\item Set $r=1-\ell$. 
\item Pick $\delta=\frac{1}{k}<\ell$, $k\in \mathbb N$.
\item Feed in a partition of $I$ into intervals. Call it $\eta$.
\item Compute $\varepsilon$ the mesh size of $\eta$.
\item Find $P_{\eta}=(P_{I_{\eta}J_{\eta}})$ where 
$$P_{I_{\eta}J_{\eta}}=\frac{\lambda(I_{\eta}\cap T^{-1} J_{\eta})}{\lambda(I_{\eta})}.$$
\item Compute the following: $H^*_{\delta,r}(P_{\eta})$, $n_1=\lceil \frac{\ln 2}{\ln r/\alpha}\rceil$, $C=r^{-n_1}$, $D=3+B$, 
$n_2=\lceil\frac{\ln 8BDCH^*_{\delta,r}(P_{\eta})}{\ln r/\alpha}\rceil $, $\gamma=\frac{\ln(r/\alpha)}{\ln(1/\alpha)}$, 
$B=\frac{1-\alpha_0+B_0}{1-\alpha}$, $\Gamma=\max\{1+\alpha_0,B_0\}.$
\item Check if $\varepsilon\le(2\Gamma)^{-1}\varepsilon_0^*(P_{\eta},\delta,r)$.\\
If (7) is not satisfied, feed in a new $\eta$ with a smaller mesh size and repeat (3)-(7); otherwise, continue.\\
\item List the eigenvalues of $P_{\eta}$ whose modulus is bigger than $r$: $\rho_{\eta,i}$, $i=1,\dots, d$. 
\item Define:
$$CL=\{\text{ all the eigenvalues from the list which are in } B(1,\delta)\}.$$
\item Check that if $\rho_{\eta,i}\notin CL$, then $\overline{B}(\rho_{\eta,i},\delta)\cap \overline{B}(1,\delta)
=\emptyset$.
\item If (10) is satisfied, report $\delta_{\text{com}}:=\delta$ and $\varepsilon_{\text{com}}:=\varepsilon$; otherwise, multiply $k$ by $2$ and repeat steps 
(2)-(11) starting with the last $\eta$ that satisfied (7).
\end{enumerate}
\end{algorithm}
\begin{proposition}\label{A}
Algorithm \ref{alg} stops after finitely many steps.
\end{proposition}
\begin{proof}
By Lemma \ref{Pro1}, for each $\delta>0$ and $r\in (\alpha,1)$ $\exists$ $\varepsilon=\text{mesh}(\eta)>0$ such that
$$\varepsilon<(2\Gamma)^{-1}\varepsilon_0^*(P_{\eta},r,\delta).$$
Therefore, the internal loop of algorithm \ref{alg} (2)-(7) stops after finitely many steps. To prove that the outer 
loop of stops after finitely many steps, observe that there exist a $K\in\mathbb N$, $K<+\infty$, such that 
$\delta=\frac{1}{K}<\min\{\ell,\overline{\delta}_0\}$, $r=1-\ell$ and $\eta$ with $\varepsilon=\text{mesh}(\eta)>0$ such that

$$\varepsilon<\min\{(2\Gamma)^{-1}\varepsilon_0(P_\eta,r,\delta),(2\Gamma)^{-1}\varepsilon_0(P,r,\delta)\}.$$
This implies $\sigma (P_{\eta})\subset V_{\delta,r}(P)\subset 
V_{\overline{\delta}_0,r}(P)$. Thus, any $P_{\eta}$ eigenvalue 
which is not in $CL$ is contained in $B(0,r)$ or it is at distance of at least $\delta$ from $B(1,\delta)$.
By Remark \ref{R}, (11) of Algorithm \ref{alg} is satisfied for this $K$.
\end{proof}
\subsection{A computer assisted bound on a hole size ensuring the existence of ACCIM}
Given the output $\varepsilon_{\text{com}}$ and $\delta_{\text{com}}$ from Algorithm \ref{alg} it
is now straightforward to prove the existence of an accim for $T_H$.  As a byproduct 
of the computation, the spectral 
information obtained from the algorithm shows 
that the associated escape rate is at most 
 $-\ln(1-\ell)$. 
\begin{theorem}\label{Th2}
Let $T_H$ be a perturbation of $T$ into an interval map with a hole. If $\lambda(H)\le\Gamma\varepsilon_{\text{com}}$ then:
\begin{enumerate}
\item $P_H$ has dominant eigenvalue $e_H>0$ whose associated eigenfunction $f_H^*\ge 0$ is the density of a $T_H$-accim;
\item $1-e_H<\delta_{\text{com}}$;
\item $1-e_H\le (1+\frac{2\alpha_0+B_0}{1-\ell-\alpha})\lambda(H)$.
\end{enumerate}
\end{theorem}
\begin{proof}
Let $\lambda(H)\le\Gamma\varepsilon_{\text{com}}$ and set $\text{mesh}(\eta)=\varepsilon_{\text{com}}$. Then by (2) of Lemma \ref{le2'} we have
$$|||P_{\eta}-P_H|||\le 2\Gamma\varepsilon_{\text{com}}\le\varepsilon_0(P_{\eta},1-\ell,\delta_{\text{com}}).$$
Using Corollary \ref{Co1} with $P_{\eta}=P_1$ and $P_H=P_2$ we obtain that $\sigma(P_H)\subset V_{\delta_{\text{com}},r}(P_{\eta})$. Now, from Algorithm \ref{alg}, recall that $B(1,\delta_{\text{com}})\cap B(0,r)=\emptyset$ and if $|\rho_i|>r$, $\rho_i$ is an eigenvalue of $P_{\eta}$, is not in $CL$, then $B(1,\delta_{\text{com}})\cap B(\rho_i,\delta_{\text{com}})=\emptyset$. Then Corollary \ref{Co1}  implies that the spectral projections of $P_{\eta}$ and $P_H$ on $B(1,\delta_{\text{com}})$ have the same rank. Hence, $P_H$ must have at least one isolated eigenvalue in $B(1,\delta_{\text{com}})$. Let $e_H$ denote the spectral radius of $P_H$. Since $P_H$ is a positive linear operator, $e_H\in\sigma(P_H)$. Moreover, $P_H$ has isolated eigenvalues in $B(1,\delta_{\text{com}})$. Thus, $e_H$ is an eigenvalue of $P_H$ and it must be in $B(1,\delta_{\text{com}})$. This ends the proof of the first two statements of the theorem. 
To prove (3) of the theorem, we first find a uniform\footnote{By uniform we mean here an upper bound which is independent of $H$. Hence it holds for all $T_H$ with $\lambda(H)\le \Gamma\varepsilon_{\text{com}}$.} upper bound on the $BV$-norm of $f^*_H$. By Lemma \ref{le2} we have
$$V(e_Hf^*_H)=VP_Hf^*_H\le\alpha Vf_H^*+(2\alpha_0+B_0)\| f^*_H\| _1.$$
Therefore,
$$
Vf^*_H\le\frac{2\alpha_0+B_0}{e_H-\alpha}\le\frac{2\alpha_0+B_0}{1-\ell-\alpha};
$$
and hence we obtain 
$$\| f^*_H\| _{BV}\le 1+\frac{2\alpha_0+B_0}{1-\ell-\alpha}.$$
Using the fact that $P$, the Perron-Frobenius operator associated with $T$, preserves integrals, we obtain
\begin{equation}
\begin{split}
1-e_{H}&=|\int_0^1Pf_H^*d\lambda-e_H\int_0^1f_H^*d\lambda|\\
&=|\int_0^1Pf_H^*d\lambda-\int_0^1P_Hf_H^*d\lambda|\\
&\le |||P-P_H|||\cdot \| f^*_H\| _{BV}\le (1+\frac{2\alpha_0+B_0}{1-\ell-\alpha})\lambda(H).
\end{split}
\end{equation} 
\end{proof}
%++++++++++++++++++++EXAMPLES +++++++++++++++++
\section{examples}\label{comp}
In this section we implement Algorithm \ref{alg} and Theorem \ref{Th2} of the previous section 
on two sample computations.  Our aim is to show the feasibility of
the computation, while at the same time, to discuss some 
analytic techniques that can be used to reduce the 
weight of computations for some of the larger matrices $P_\eta$ that 
may arise during application of Algorithm \ref{alg}. Large matrices should be expected when $\alpha$ is close to 1 or 
alternatively, when the escape rate tolerance $\ell$ is small. 
We will take advantage of the second mechanism; 
 in both examples we use the same map $T$:
 \begin{equation*}
T(x)=\left\{\begin{array}{cc}
\frac{9x}{1-x}&\mbox{for $0\le x\le \frac{1}{10}$}\\
10x-i&\mbox{for $\frac{i}{10}<x\le\frac{i+1}{10}$}
\end{array}
\right. ,
\end{equation*}
where $i=1,2,\dots ,9$. 
However, in the first example $\ell=1/25$ and in the second example $\ell=1/40$; i.e., in second example we will be looking 
for the size of a hole which guarantees the smaller escape rate. 
We remark none of our computations are
particulary time consuming\footnote{In particular, creating an 
Ulam matrix of size $5000\times 5000$, or even much bigger, is not really time demanding. Once a computer code is 
developed for this purpose, which only requires the formula of the map and the number 
of bins of the Ulam partition as an input, it will excute the nonzero entries in few minutes if not less.}
except for the computation of an upper bound on $H^*_{\delta,r}(P_\eta)$\footnote{In our computations we found a rigorous upper 
bound on 
$H^*_{\delta,r}(P_{\eta})$ for an Ulam matrix of size $5000\times 5000$. This computation took few hours using MATLAB on a desktop 
computer.}.  We now turn to the computations.

The Lasota-Yorke inequality for $P$ is given by:
$$VPf\le 1/9Vf+2/9\| f\| _1.$$ Therefore
$P$ satisfies (A1) with  $\alpha_0=1/9$ and $B_0=2/9$ and consequetly, $\Gamma=10/9$, $\alpha=1/3$, $B=5/3$ and  $D=A(A+B+2)=14/3$. 
\begin{example}\label{Ex1}
Given $\ell=1/25$, using Algorithm \ref{alg}, we show that if $\lambda(H)\in(0,\frac{20}{9}\times 10^{-4}]$, $T_H$ has 
an accim with escape rate $-\ln e_H < -\ln(24/25)$.  The values of the variables
\footnote{All these variables depend on $r$ and $\delta$.  $H^{*}_{\delta,r}$ and $n_2$ also depend on $\varepsilon=\text{mesh}(\eta )$ .} involved in the computation are summarized in Table 1.

\bigskip

We present here the method which we have followed to rigorously compute an upper bound on $H^*_{\delta,r}(P_{\eta})$, for $\text{mesh}(\eta)=2\times 10^{-4}$. Using MATLAB we found the dominant eigenvalue $1$ of $P_{\eta}$ is simple and that there are no other peripheral eigenvalues. Moreover, the modulus of any non-peripheral eigenvalues is smaller than $\alpha_0=1/9$.  Therefore we have the following estimate (see \cite{Ka})
\begin{equation}\label{eq3:ex}
\| (z-P_{\eta})^{-1}\| _1\le\delta^{-1}\| \Pi_{1}\| _1 +\| R(z)\| _1,
\end{equation} 
where $\| \Pi_{1}\| $ is the projection associated with the eigenvalue $1$ of the operator $P_{\eta}$, and $R(z)$ is the resolvent of the 
operator $P_{\eta}({\bf 1}-\Pi_{1})$. Since $|z|>r=1-\ell >\alpha_0$, $R(z)$ can be represented by a convergent Neumann series. Indeed, we have 
\begin{equation}\label{ex4:eq}
\begin{split}
\| R(z)\| _1 & = \| \sum_{n=0}^{\infty}\frac{\left(P_{\eta}({\bf 1}-\Pi_{1})\right)^n}{z^{n+1}}\| _1\\
&\le\frac{1}{r}\left(\sum_{n=0}^5\frac{\| \left(P_{\eta}({\bf 1}-\Pi_{1})\right)^n\| _1}{r^n}+
\sum_{n=6}^{\infty}\frac{\| \left(P_{\eta}({\bf 1}-\Pi_{1})\right)^n\| _1}{r^n}\right)\\
&\le \frac{1}{r}\left[\sum_{n=0}^5\frac{\| \left(P_{\eta}({\bf 1}-\Pi_{1})\right)^n\| _1}{r^n}\left(1+\sum_{m=1}^{\infty}\left(\frac{\| \left(P_{\eta}({\bf 1}-\Pi_{1})\right)^{6}\| _1}{r^{6}}\right)^m\right)\right]\\
&= 7.444310493.
\end{split}
\end{equation}
The computation of the estimate in (\ref{ex4:eq}) is the most time consuming step
in the algorithm\footnote{Precisely, the work is in the computation of the powers $\left(P_{\eta}({\bf 1}-\Pi_{1})\right)^n $, $n=1,\dots ,6$. Once these powers are
known the computation of the norm is very fast.  However, we will see in the next example how we can benefit from these numbers and avoid time consuming computations when dealing with a higher order Ulam approximation in the case of a smaller hole.}. Using the definition of $H^*_{\delta,r}$ and inequality (\ref{eq3:ex}), we obtain that 
$$H^*_{\delta,r}\le 45.46070939.$$ 
\end{example}
{\small
\begin{table}[h]\label{table1}
\begin{center}
\begin{tabular}{|c|c|}
  \hline
  % after \\: \hline or \cline{col1-col2} \cline{col3-col4} ...
  $r$ & 24/25\\
  \hline
  $\delta$ & 1/26\\
  \hline
  $\varepsilon$ &$2\times 10^{-4}$\\
  \hline
  $H^{*}_{\delta,r}$ & 45.46070939\\
   \hline
    $n_1$ & 1\\
    \hline
    $C$& 25/24\\
    \hline
    $n_2$& 8\\
    \hline
    $(2\Gamma)^{-1}\varepsilon_0^*$& 0.0002319492040\\ 
   \hline
$\text{Loop I}$&\text{Pass}\\
\hline
\text{Loop II}& \text{Pass}\\
\hline
$\text{Output I}$& $\varepsilon_{\text{com}}=2\times 10^{-4}, \delta_{\text{com}}=1/26$\\
\hline
$\text{Output II}$& $\lambda(H)\in(0,\frac{20}{9}\times 10^{-4}]\implies\, T_H \text{ admits an accim } \mu$\\ 
\text{ }&\text{with escape rate }$-\ln e_H< -\ln (24/25)$\\
\hline
\end{tabular}
\end{center}
 \caption {The output of Algorithm \ref{alg} for $l=1/25$}
 \end{table}
 }
\begin{example}\label{Ex2}
Given $\ell=1/40$, using Algorithm \ref{alg}, we show that if $\lambda(H)\in(0,\frac{10}{9}\times 10^{-5}]$, $T_H$ has 
an accim with escape rate $-\ln e_H < -\ln(39/40)$. The values of the variables involved in the computations are 
summarized in Table 2. 

\bigskip

Here, we explain how we have obtained some of the values which appear in Table 2. In particular, we will explain how we have avoided time demanding 
computation of $H^*_{\delta,r}(P_{\eta'})$ and rigorously estimated 
$$H_{\delta,r}(P_{\eta'})\le 1036.693385,$$
where $\text{mesh}(\eta')\le \text{mesh}(\eta) =2\times 10^{-4}$. In the first pass 
through the Algorithm \ref{alg}, we start with $\varepsilon=\text{mesh}(\eta)=2\times 10^{-4}$. This is the same $\varepsilon$ which closed Algorithm \ref{alg} in Example \ref{Ex1}. So the numbers $\| (P_{\eta}({\bf 1}-\Pi_1))^n\| _1,\, n=1,\dots,6$, can be obtained from the computation of Example \ref{Ex1} as they do not depend on $r$ and $\delta$. Thus, the rigorous estimate
$$H^*_{\delta,r}(P_{\eta})\le 63.73181657$$
easily follows.  However, the inner loop of Algorithm \ref{alg} will fail because
$$2\times 10^{-4}=\varepsilon=\text{mesh}(\eta)>(2\Gamma)^{-1}\varepsilon_0^*=0.0001763820641.$$
Next, Algorithm \ref{alg} asks us to feed another Ulam partition $\eta'$ with $\text{mesh}(\eta')<2\times 10^{-4}$ and 
to repeat the inner loop of Algorithm \ref{alg}. Here, we have used a 3-step trick
 to avoid time demanding estimate of the new value 
$H^*_{\delta, r}(P_{\eta'})$:
\begin{enumerate}
\item Let us suppose for a moment that we are only concerned with a rigorous approximation of the spectrum of $P$, the 
operator associated with $T$. Then $P, P_{\eta}$ and $P_{\eta'}$ satisfy a common Lasota-Yorke inequality which does not 
involve $\alpha$, but rather $\alpha_0$ (see Remark \ref{Re:LY}).
Now, re-checking the computations which were obtained in the first run of Algorithm \ref{alg} and this time with $\alpha\equiv\alpha_0=1/9$ and the modfications of $B$ to $\hat B=1+ \frac{B_0}{1-\alpha_0}$ and $D$ to 
$\hat D=3+\hat B$. Then the value of $(2\Gamma)^{-1}\varepsilon^*_0$ changes to
$$(2\Gamma)^{-1})\varepsilon_0^*=0.0002425063815> 2\times 10^{-4}=\text{mesh}(\eta).$$
Consequently, for any $\eta'$ with $\text{mesh}(\eta')\le\text{mesh}(\eta)$, we have
\begin{equation*}
\begin{split}
|||P_{\eta}-P_{\eta'}|||&\le |||P_\eta-P|||+|||P-P_{\eta'}|||\\
&\le \Gamma\text{mesh}(\eta)+\Gamma\text{mesh}(\eta')\\
&\le 2\Gamma\text{mesh}(\eta)<\varepsilon_0^*(P_{\eta},r,\delta).
\end{split}
\end{equation*}  
Therefore, we can use part one of Theorem \ref{Th1} with $P_1=P_{\eta}$ and $P_2=P_{\eta'}$.
\item In particular, for any $z\in\mathbb C\setminus V_{\delta,r}(P_{\eta'})$ we have
$$\| (z-P_{\eta'})^{-1}\| _{BV}\le \frac{4(1+\hat B)}{1-r}r^{-n_1}+\frac{1}{2\varepsilon_1}\le 1036.693385.$$
Recall that $\varepsilon_1=\frac{r^{n_1+n_2}}{8\hat B(H_{\delta,r}(P_{\eta})+\frac{1}{1-r})}$.
\item Now we go back to the problem of finding the size of a hole which guarantees the existence of a $T_H$-accim with the desired escape rate. Here 
$\alpha=3\alpha_0$ and all we have to do is to feed the estimate on $H_{\delta,r}(P_{\eta})$ obtained in Step 2, 
together with the new $n_2$, in the formula of $\varepsilon_0$ to obtain that
$$(2\Gamma)^{-1}\varepsilon_0\ge 0.00001216687545.$$
Hence, we can deduce that $\text{mesh}(\eta')=10^{-5}$ will do the job; i.e., 
$\lambda(H)\in(0,\frac{10}{9}\times 10^{-5}] 
\implies\, T_H \text{ admits an accim } \mu$ with escape rate $-\ln e_H< -\ln (39/40)$.
\end{enumerate}  
\end{example}
{\small
\begin{table}[h]\label{table2}
\begin{center}
\begin{tabular}{|c|c|c|}
  \hline
  % after \\: \hline or \cline{col1-col2} \cline{col3-col4} ...
  $r$ & 39/40 & 39/40\\
  \hline
  $\delta$ & 1/41 & 1/41\\
  \hline
  $\varepsilon$ &$2\times 10^{-4}$ & $10^{-5}$\\
  \hline
  $\text{Upper bound on }H_{\delta,r}$ & 63.73181657 & 1036.693385\\
   \hline
    $n_1$ & 1 & 1\\
    \hline
    $C$& 40/39 &40/39\\
    \hline
    $n_2$& 8& 11\\
    \hline
    $\text{Lower bound on }(2\Gamma)^{-1}\varepsilon_0$& 0.0001763820641& 0.00001216687545\\ 
   \hline
$\text{Loop I}$& \text{Fail: reduce $\varepsilon$} &\text{Pass}\\
\hline
\text{Loop II}& \text{}& \text{Pass}\\
\hline
$\text{Output I }$&\text{ }& $\varepsilon_{\text{com}}=10^{-5}, \delta_{\text{com}}=1/41$\\
\hline
$\text{Output II}$&\text{ }& $\lambda(H)\in(0,\frac{10}{9}\times 10^{-5}] 
\implies\, T_H \text{ admits an accim } \mu$\\ 
\text{ }&\text{ }&\text{with escape rate }$-\ln e_H< -\ln (39/40)$\\
\hline
\end{tabular}
\end{center}
 \caption {The output of Algorithm \ref{alg} for $l=1/40$}
 \end{table} 
 } 
 \section{The effect of the position of a hole}\label{smooth}
 The results of the previous section give upper bounds on the 
 escape rate that are uniform for a given size of hole, independent
 of the position of the hole.  However, 
  it has been observed already in \cite{BY} that the position of the hole can affect the escape rate; i.e., given a map $T$ and two holes $H_1$, $H_2$, with $\lambda(H_1)=\lambda(H_2)$, it may happen that the escape through $H_1$, say, may be bigger than the escape rate through $H_2$. For example, 
  define the sets
 $$\text{Per}(H_i)=\{p:\, p\in\mathbb N \text{ s.t. for some } x\in H_i,\, T^p(x)=x, \text{ and } T^{p-1}(x)\not= x \};\, i=1,2,$$  
For certain maps, if 
 $$\textnormal{Minimum}\{p\in\text{Per}(H_1)\}\le\textnormal{Minimum}\{p\in\text{Per}(H_2)\}$$
 then the escape rate through $H_1$
 will be smaller than the escape rate through $H_2$ . 
 
 \bigskip
 
 In \cite{KL2} Keller and Liverani obtained  
 precise asymptotic information about 
 the effect of the location 
 of the hole.  Roughly speaking, for a system of holes shrinking to a
 single point, the rate of decay of escape rate depends on two things:
 the value of the invariant density of the map $T$ at the point the holes shrink
 to, and whether or not this point is periodic. 
  We now state a version of this result and will discuss in the next subsection, in a smooth setting, how a combination of Algorithm \ref{alg}, with the proper modification of $P_{\eta}$, can be used with this theorem when the formula of the invariant density of $T$ cannot be found explicitly. When smoothness is not assumed, as in this paper,  obtaining  asymptotics for the escape rate relative to the 
 size of the hole appears to be an 
 open problem.
  \begin{theorem}\label{Th:KL}\cite{KL2}
Let $T$ be piecewise $C^2$ on a finite partition of $[0,1]$ and assume it is piecewise expanding and mixing. Let $\{H_{\kappa}\}$ be a sequence of holes such that $H_{\kappa}\supset H_{\kappa'}$, for $0\le \kappa'<\kappa$, with $H_0=\{y\}$ for some point $y\in[0,1]$ which is a point of continuity of both $T$ and $f^*$, $f^*$ is the invariant density of $T$. Let $T_{H_{\kappa}}$ be a perturbation of $T$ into a map with a hole. Assume that $\inf f^*_{|H_{\kappa}}>0$.  For $\lambda(H_{\kappa})$ sufficiently small\footnote{We can of course quantify what we mean by sufficiently small using Algorithm \ref{alg}.} we have:
 \begin{enumerate}
 \item If $y$ is non-periodic then $\lim_{\kappa\to 0} \frac{1-e_{H_{\kappa}}}{\lambda(H_{\kappa})}=f^*(y)$.
 \item If $y$ is periodic with period $p$ then $\lim_{\kappa\to 0} \frac{1-e_{H_{\kappa}}}{\lambda(H_{\kappa})}=f^*(y)\left(1-\frac{1}{|(T^p)'(y)|}\right).$
 \end{enumerate}
 \end{theorem}
 \subsection{ $C^3$ circle maps}
 Theorem \ref{Th:KL} requires the knowledge of the value of the invariant density $f^*$; in particular, its value at the point $y$. Unfortunately, the approximate invariant density which is obtained by Ulam's method in Algorithm \ref{alg} does not provide a pointwise approximation of $f^*$. However, in a smooth setting, one can modify Ulam's scheme, and the function spaces where $P$ and $P_{\eta}$ act, to obtain rigorous approximation of $\| f^*-f_{\eta}\| _{\infty}$, where $f_{\eta}$ denotes the fixed function for the modified finite rank operator $P_{\eta}$. Of course our main theoretical tool, 
Theorem \ref{Th1} will need to be modified. 
In \cite{KL} an abstract version of Theorem \ref{Th1} is proved.  The result 
requires bounded linear operators $P_1$ and $P_2$ acting on two abstract Banach spaces whose norms $\| \cdot\| $ and $|\cdot|$ satisfy $|\cdot|\le \| \cdot\| $, and the unit ball of $\| \cdot\| $ is $|\cdot|$-compact.  Thus, Theorem \ref{Th1} is a particular application of the general result of \cite{KL}.
When $T$ is a $C^3$ circle map, it is well known that its invariant density $f^*$ is $C^2$. Thus, instead of $L^1$ and $BV$, one can study the action of $P$ and a smooth version\footnote{For instance one can use a piecewise linear approximation method \cite{DL}.} of $P_{\eta}$ on the spaces $W^{1,1}$ and $W^{1,2}$ with norms
 $$\| f\| _{W^{1,1}}=\| f\| _1+\| f'\| _1$$
 $$\| f\| _{W^{1,2}}=\| f\| _1+\| f'\| _1+\| f''\| _1$$
 respectively. A common Lasota-Yorke inequality of $P$ and $P_{\eta}$ in this setting is given by: for $f\in W^{1,2}$ and $n\in\mathbb N$ we have
 $$\| Pf\| _{W^{1,2}}=\alpha_0^{2n}\| f\| _{W^{1,2}}+{\bar B}\| f\| _{W^{1,1}},$$ 
  $$\| P_{\eta}f\| _{W^{1,2}}=\alpha_0^{2n}\| f\| _{W^{1,2}}+{\bar B}\| f\| _{W^{1,1}},$$ 
where $\bar B\ge 0$ which depends on $T$ only. Using this setting, one obtains the estimate
$$\| f^*-f_{\eta}\| _{W^{1,1}}\le\bar C\cdot \text{mesh}(\eta).$$  
For more details and for a proof of the above Lasota-Yorke inequality we refer to Section 10.2 of \cite{Li}. 
 
 \bigskip
 
 In a setting like this, one can then repeat Algorithm \ref{alg} with the smooth version of $P_{\eta}$ and obtain the following reformulation of Theorem \ref{Th:KL} to a setting where the invariant density $f^*$ is a priori unknown. Note that for $f\in W^{1,2}$, $\| f\| _{\infty}\le\| f\| _{W^{1,1}}$.
 \begin{theorem}\label{Th3}
Let $T$ be a $C^3$ circle map. Let $\{H_{\kappa}\}$ be a sequence of holes such that $H_{\kappa}\supset H_{\kappa'}$, for $0\le \kappa'<\kappa$, with $H_0=\{y\}$ for some point $y\in[0,1]$. Let $f^*$ be the invariant density of $T$, and $T_{H_{\kappa}}$ be a perturbation of $T$ into a map with a hole\footnote{$\inf_{x\in[0,1]}f^*>0$ for $C^2$ circle maps. See \cite{KZ} or \cite{Mu}. Thus, the assumption $\inf f^*_{|H_{\kappa}}>0$ is automatically satisfied for such maps.}. Let $\varepsilon=\text{mesh}(\eta)$. $\exists$ a constant $\bar C=\bar C(P_{\eta})$ such that for $\lambda(H_{\kappa})\in (0,\Gamma\varepsilon_{\text{com}}]$, we have:
 \begin{enumerate}
 \item If $y$ is non-periodic then 
 $$f_{\eta}(y)- \bar C\cdot \varepsilon \le \lim_{\kappa\to 0} \frac{1-e_{H_{\kappa}}}{\lambda(H_{\kappa})}\le
  f_{\eta}(y)+ \bar C\cdot \varepsilon.$$
 \item If $y$ is periodic with period $p$ then 
 \begin{equation*}
\begin{split}
\left(f_{\eta}(y)- \bar C\cdot\varepsilon \right)\left(1-\frac{1}{|(T^p)'(y)|}\right)&\le\lim_{\kappa\to 0} \frac{1-e_{H_{\kappa}}}{\lambda(H_{\kappa})}\le\\
&\left(f_{\eta}(y)+ \bar C\cdot\varepsilon \right)\left(1-\frac{1}{|(T^p)'(y)|}\right).
 \end{split}
 \end{equation*}
 \end{enumerate}
 \end{theorem}
 \begin{proof}
 We only give a sketch of the proof. Suppose that we have used Algorithm \ref{alg} with the proper modification of $P_{\eta}$ and the function spaces. Then the invariant density, which is a byproduct of the algorithm, would provide the following estimate:
 $$\| f_{\eta}-f^*\| _{\infty}\le \bar C\cdot\varepsilon.$$
 Consequently, for any $y\in [0,1]$, we have 
 \begin{equation}\label{eq:app}
 | f_{\eta}(y)-f^*(y)|\le  \bar C\cdot \varepsilon. 
 \end{equation} 
 \end{proof}
 Thus, the proof follows by using (\ref{eq:app}) and Theorem \ref{Th:KL}.
 \begin{remark}
 All the constants which are hiding in the computation of $\bar C=\bar C(P_{\eta})$ can be rigorously computed using Theorem \ref{Th1} with the spaces $W^{1,1}$ and $W^{1,2}$. It should be pointed out that these constants cannot be computed if one attempts to do this a approximation in the $L^1$, $BV$ framework. This is because the estimates will depend on $(f^*)''$ which is a priori unknown. 
 \end{remark}
 \begin{remark}
 For $C^2$ Lasota-Yorke maps, a result similar to Theorem \ref{Th3} is not obvious at all. The problem for $C^2$ Lasota-Yorke maps involves two issues: 
 \begin{enumerate}
 \item The invariant density $f^*$ is $C^1$. This means that the framework of $W^{1,1}$, $W^{1,2}$ cannot be used.
 \item If one uses the function spaces $BV$ and $L^1$, then to the best of our knowledge, only the original Ulam method will fit in this setting. The problem with Ulam's method is that it provides only good estimates in the $L^1$ norm $\| f^*-f_{\eta}\| _1= \bar C\cdot \varepsilon\ln1/\varepsilon$. However, typically, $\| f^*-f_{\eta}\| _{BV}\not\to 0$.
 \end{enumerate}
 Our last comment on this is that providing a scheme for $C^2$ Lasota-Yorke maps to obtain a result similar to that of Theorem \ref{Th3} would be an interesting problem. 
 \end{remark}
\bibliographystyle{amsplain}

\end{document}